\documentclass{amsart}

\usepackage[latin1]{inputenc}
\usepackage[english]{babel}
\usepackage{indentfirst}
\usepackage{amssymb}
\usepackage{amsthm}
\usepackage{xcolor}
\usepackage[all]{xy}
\usepackage[mathscr]{eucal}

\newcommand{\N}{\mathbb{N}}
\newcommand{\erre}{\mathbb{R}}
\newcommand{\sub}{\subseteq}

\newcommand\Xs{X^\ast}

\newcommand\xs{x^\ast}

\newcommand{\Ys}{Y^{\ast}}
\newcommand\ys{y^\ast}

\def\epsilon{\varepsilon}
\newtheorem{theo}{Theorem}[section]
\newtheorem{lem}[theo]{Lemma}%[section]
\newtheorem{pro}[theo]{Proposition}%[section]
\newtheorem{cor}[theo]{Corollary}%[theo]
\newtheorem{fact}[theo]{Fact}%[section]
\newtheorem{rem}[theo]{Remark}

\numberwithin{equation}{section}

\begin{document}

%\date{\today}

\title[Isometric factorization of vector measures\dots]{Isometric factorization of vector measures and applications to spaces of integrable functions}
\author[Nygaard]{Olav Nygaard}
\author[Rodr\'{i}guez]{Jos\'{e} Rodr\'{i}guez}

\address{Department of Mathematics\\ University of Agder\\ Servicebox 422\\
4604 Kristiansand\\ Norway}
\email{Olav.Nygaard@uia.no}
\urladdr{https://home.uia.no/olavn/}

\address{Dpto. de Ingenier\'{i}a y Tecnolog\'{i}a de Computadores\\Facultad de Inform\'{a}tica\\
Universidad de Murcia\\ 30100 Espinardo (Murcia)\\ Spain} 
\email{joserr@um.es}
\urladdr{https://webs.um.es/joserr/}

\thanks{The research of Jos\'{e} Rodr\'{i}guez was partially supported by {\em Agencia Estatal de Investigaci\'{o}n} [MTM2017-86182-P, grant cofunded by ERDF, EU] 
and {\em Fundaci\'on S\'eneca} [20797/PI/18]}

\subjclass[2010]{46B20; 46G10}%
\keywords{Vector measure; finite variation; space of integrable functions; integration operator; Davis-Figiel-Johnson-Pelcz\'{y}nski factorization}%
%

% ----------------------------------------------------------------
\begin{abstract}
Let $X$ be a Banach space, $\Sigma$ be a $\sigma$-algebra, and $m:\Sigma\to X$ be a (countably additive) vector measure.
It is a well known consequence of the Davis-Figiel-Johnson-Pelcz\'{y}nski factorization procedure that 
there exist a reflexive Banach space~$Y$, a vector measure $\tilde{m}:\Sigma \to Y$
and an injective operator $J:Y \to X$ such that $m$ factors as $m=J\circ \tilde{m}$.
We elaborate some theory of factoring vector measures and their integration operators with the help of the 
isometric version of the Davis-Figiel-Johnson-Pelcz\'{y}nski factorization procedure.
Along this way, we sharpen a result of Okada and Ricker that if the integration operator on $L_1(m)$ is weakly compact, then 
$L_1(m)$ is equal, up to equivalence of norms, to some $L_1(\tilde m)$ where $Y$ is reflexive; here 
we prove that the above equality can be taken to be isometric.  
\end{abstract}
\maketitle

\date{\today}

%\thanks{}

\maketitle

\section{Introduction}

Let throughout $(\Omega,\Sigma)$ be a measurable space, $X$ be a real Banach 
space, and $m:\Sigma \to X$ be a countably additive vector measure 
(not identically zero). Let us agree that $m$ being a vector measure automatically means that $m$ is countably additive and defined on some 
$\sigma$-algebra of subsets of some set.

The range of~$m$, i.e., the set $\mathcal{R}(m):=\{m(A):A\in \Sigma\}$ is relatively weakly compact by a classical result of Bartle, Dunford and Schwartz
(see, e.g., \cite[p.~14, Corollary~7]{die-uhl-J}).
So, the Davis-Figiel-Johnson-Pelcz\'{y}nski (DFJP) factorization method~\cite{dav-alt} applied to the closed absolute convex hull
$\overline{{\rm aco}}(\mathcal{R}(m))$ of~$\mathcal{R}(m)$ ensures
the existence of a reflexive Banach space~$Y$ and an injective operator $J:Y \to X$ such that $J(B_Y) \supseteq \mathcal{R}(m)$. Accordingly, $m$
factors as $m=J\circ \tilde{m}$ for some map $\tilde{m}:\Sigma\to Y$ which turns out to be a vector measure as well
(cf., \cite[Theorem~2.1(i)]{rod15}). In commutative diagram form:
$$%\begin{equation}\label{eqn:diag}
	\xymatrix@R=3pc@C=3pc{\Sigma
	\ar[r]^{m} \ar[d]_{\tilde{m}} & X\\
	Y  \ar@{->}[ur]_{J}  & \\
	}
$$%\end{equation}
Note that the vector measure~$\tilde{m}$ need not have finite variation although $m$ has finite variation. Indeed, if $m$
does not have a Bochner derivative with respect to~$|m|$, then neither does $\tilde{m}$ (since $J$ is an operator) 
and so $\tilde{m}$ does not have finite variation, because $Y$ has the Radon-Nikod\'{y}m property
and $\tilde{m}$ is $|m|$-continuous (by the injectivity of~$J$). However, if $m$ has finite variation
and a Bochner derivative with respect to~$|m|$, then $m$ factors via a vector measure of finite variation taking values
in a separable reflexive Banach space. This result is implicit in the proof of \cite[Theorem~5.2]{rodpiazza-1} and we 
include it as Theorem~\ref{theo:RP} for the reader's convenience.

The previous way of factoring~$m$ is equivalent to applying the DFJP method to the (weakly compact) integration operator 
on the Banach lattice $L_\infty(m)$, i.e., 
$$
	I^{(\infty)}_m:L_\infty(m)\to X,
	\quad
	I^{(\infty)}_m(f):=\int_\Omega f\, dm,
$$ 
because one has ${\rm aco}(\mathcal{R}(m)) \sub I^{(\infty)}_m(B_{L_\infty(m)}) \sub 2\overline{{\rm aco}}(\mathcal{R}(m))$
(see, e.g., \cite[p.~263, Lemma~3(c)]{die-uhl-J}).
But there is still another approach which is based on factoring the integration operator on the (larger) Banach lattice $L_1(m)$
of all real-valued $m$-integrable functions defined on~$\Omega$, i.e.,
$$
	I_m:L_1(m)\to X,
	\quad
	I_m(f):=\int_\Omega f\, dm.
$$ 
Note, however, that $I_m$ need not be weakly compact and so in this case the DFJP method
gives a factorization through a non-reflexive space. The DFJP factorization was already applied to~$I_m$
in \cite{oka-ric} and \cite{rod15}. Okada and Ricker showed that if $I_m$ is weakly compact, then 
there exist a reflexive Banach space~$Y$, a vector measure $\tilde{m}:\Sigma \to Y$
and an injective operator $J:Y \to X$ such that $m=J\circ \tilde{m}$ and
$L_1(m)=L_1(\tilde{m})$ with {\em equivalent} norms (see \cite[Proposition~2.1]{oka-ric}).

In this paper we study the factorization of vector measures and their integration operators
with the help of the {\em isometric} version of the DFJP procedure developed by Lima, Nygaard and Oja~\cite{lim-nyg-oja}
(DFJP-LNO for short). 

The paper is organized as follows. 
In Section~\ref{section:preliminaries} we include
some preliminaries on spaces of integrable functions with respect to a vector measure and their integration operators,
as well as on the DFJP-LNO method.

In Section~\ref{section:lemmas} we obtain some results on factorization of integration operators that are a bit more general than applying the 
DFJP factorization procedure directly.

In Section~\ref{section:Iinfty} we present our main results. Theorem~\ref{theo:Iinfty} collects
some benefits of applying the DFJP-LNO factorization to $I^{(\infty)}_m$. Thus, one gets 
a reflexive Banach space~$Y$, a vector measure $\tilde{m}:\Sigma\to Y$ with $\|m\|(\Omega)=\|\tilde{m}\|(\Omega)$ and an injective
norm-one operator $J:Y \to X$ such that $I^{(\infty)}_m = J \circ I^{(\infty)}_{\tilde{m}}$.
Moreover, the special features of the DFJP-LNO factorization also provide the following interpolation type inequality:
\begin{equation}\label{eqn:intro}
	\|I^{(\infty)}_{\tilde{m}}(f)\|^2 \leq C \|m\|(\Omega) \|f\|_{L_\infty(m)} \|I^{(\infty)}_m(f)\|
	\quad\text{for all $f\in L_\infty(m)$,}
\end{equation}
where $C>0$ is a universal constant. As a consequence, $I^{(\infty)}_{\tilde{m}}$ factors through 
the Lorentz space $L_{2,1}(\|m\|)$ associated to the semivariation of~$m$ (Proposition~\ref{pro:Pisier}).

Theorem~\ref{theo:Im} gathers some consequences of the DFJP-LNO method
when applied to~$I_m$. In this case, one gets a factorization as follows:
$$
	\xymatrix@R=3pc@C=3pc{L_1(m)
	\ar[r]^{I_m} \ar[d]_{I_{\tilde{m}}} & X\\
	Y  \ar@{->}[ur]_{J}  & \\
	}
$$
where $Y$ is a (not necessarily reflexive) Banach space, $\tilde{m}:\Sigma\to Y$ is a vector measure and $J$ is an injective
norm-one operator. Now, the equality 
$$
	L_1(m)=L_1(\tilde{m})
$$
holds with {\em equal} norms.
Moreover,
an inequality similar to~\eqref{eqn:intro} is the key to prove that 
$\tilde{m}$ has finite variation (resp., finite variation and a Bochner derivative with respect to it) whenever $m$ does.
As a particular case, we get the isometric version of the aforementioned result of Okada and Ricker (Corollary~\ref{cor:Iweaklycompact}).

\section{Preliminaries}\label{section:preliminaries}

By an {\em operator} we mean a continuous linear map between Banach spaces.
The topological dual of a Banach space~$Z$ is denoted by~$Z^*$. 
We write $B_Z$ to denote the closed unit ball of~$Z$, i.e., $B_{Z}=\{z\in Z:\|z\|\leq 1\}$. 
The absolute convex hull (resp., closed absolute convex hull) of a set $S\sub Z$ is denoted by ${\rm aco}(S)$
(resp., $\overline{{\rm aco}}(S)$).

Our source for basic information on vector measures is \cite[Chapter~I]{die-uhl-J}.  
The symbol $|m|$ stands for the {\em variation} of~$m$, while its {\em semivariation} 
is denoted by~$\|m\|$. We write $x^*m$ to denote the composition of $x^*\in X^*$ and~$m$.
A set $A\in \Sigma$ is said to be {\em $m$-null} if $\|m\|(A)=0$ or, equivalently, $m(B)=0$ for every $B\in \Sigma$ with~$B\sub A$.
The family of all $m$-null sets is denoted by $\mathcal{N}(m)$. A {\em control measure} of~$m$
is a non-negative finite measure~$\mu$ on~$\Sigma$ such that $m$ is $\mu$-continuous, i.e., $\mathcal{N}(\mu) \sub \mathcal{N}(m)$; if 
$\mu$ is of the form $|x^*m|$ for some $x^*\in X^*$, then it is called a {\em Rybakov control measure}.
Such control measures exist for any vector measure (see, e.g., \cite[p.~268, Theorem~2]{die-uhl-J}).

\subsection{$L_1$-spaces of vector measures and integration operators}

A suitable reference for basic information on $L_1$-spaces of vector measures is \cite[Chapter~3]{oka-alt}. 
A $\Sigma$-measurable function $f:\Omega \to \erre$ is called {\em weakly $m$-integrable} if 
$\int_\Omega |f|\,d|\xs m| <\infty$ for every $\xs\in\Xs$. In this case, for each $A\in \Sigma$ there is $\int_A f \, dm\in X^{**}$
such that
\[
	\Big ( \int_A f \, dm\Big)(\xs)=\int_A f\,d(\xs m)
	\quad\text{for all $x^*\in X^*$}.
\]
By identifying functions which coincide $m$-a.e., the set $L_1^w(m)$ of all weakly $m$-integrable functions  
forms a Banach lattice with the $m$-a.e. order and the norm
\[
	\|f\|_{L^w_1(m)}:=\sup_{\xs\in B_{\Xs}}\int_\Omega |f|\,d|\xs m|.
\]

Given any Rybakov control measure~$\mu$ of~$m$, the space $L^w_1(m)$ {\em embeds continuously} into~$L_1(\mu)$, i.e.,
the identity map $L^w_1(m) \to L_1(\mu)$ is an injective operator. From this embedding one gets the following well known property:

\begin{fact}\label{fact:convergence}
Let $(f_n)$ be a sequence in $L^w_1(m)$ which converges in norm to some $f\in L^w_1(m)$. Then 
there is a subsequence $(f_{n_k})$ which converges to~$f$ $m$-a.e.
\end{fact}

A $\Sigma$-measurable function $f:\Omega \to \mathbb{R}$ is said to be $m$-integrable if it is weakly $m$-integrable
and $\int_A f\, dm\in X$ for all $A\in \Sigma$.
The closed sublattice of $L_1^w(m)$ consisting of all $m$-integrable functions is denoted by~$L_1(m)$. 
The Banach lattice $L_1(m)$ is order continuous and has a weak order unit (the function $\chi_\Omega$). 
The following result
of Curbera (see \cite[Theorem~8]{cur1}) makes the class of $L_1(m)$-spaces 
extremely interesting: {\em if $E$ is an order continuous Banach lattice with a weak order unit, then there exists 
an $E$-valued positive vector measure~$m$ such that $L_1(m)$ and $E$ are lattice isometric.} (A vector measure taking values in a Banach lattice~$E$
is said to be {\em positive} if its range is contained in the positive cone of~$E$.)

We write ${\rm sim}\, \Sigma$ to denote the set of all {\em simple functions} from~$\Omega$ to~$\erre$. 
Just as for scalar $L_1$-spaces, ${\rm sim}\, \Sigma$ is a norm-dense linear subspace of $L_1(m)$. Note that
$\int_\Omega \chi_A \, dm=m(A)$ and $\|\chi_A\|_{L_1(m)}=\|m\|(A)$ for all $A\in \Sigma$.

Any $m$-essentially bounded $\Sigma$-measurable function $f:\Omega \to \erre$ is $m$-integrable. 
By identifying functions which coincide $m$-a.e., the set $L_\infty(m)$ of all
$m$-essentially bounded $\Sigma$-measurable functions is a Banach lattice with the $m$-a.e. order and the $m$-essential sup-norm. Of course,
$L_\infty(m)$ is equal to the usual space $L_\infty(\mu)$ for any Rybakov control measure $\mu$ of~$m$.
It is known (see, e.g., \cite[Proposition~3.31]{oka-alt}) that: 
\begin{itemize}
\item[(i)] if $g\in L_\infty(m)$ and $f\in L_1(m)$, then $f g \in L_1(m)$ and 
$$
	\|fg\|_{L_1(m)}\leq \|f\|_{L_1(m)}\|g\|_{L_\infty(m)};
$$ 
\item[(ii)] the identity map
$\alpha_\infty:L_\infty(m)\to L_1(m)$ is an (injective) weakly compact operator with $\|\alpha_\infty\|=\|m\|(\Omega)$. 
\end{itemize}

The following formula
for the norm on~$L_1(m)$ will also be useful (see, e.g., \cite[Lemma~3.11]{oka-alt}):

\begin{fact}\label{fact:norming}
For every $f\in L_1(m)$ we have
$$
	\|f\|_{L_1(m)}=\sup_{g\in B_{L_\infty(m)}}
	\Big\| \int_\Omega fg \, dm \Big\|=
	\sup_{g\in B_{L_\infty(m)}\cap {\rm sim}\, \Sigma}
	\Big\| \int_\Omega fg \, dm \Big\|.
$$
\end{fact}

A fundamental tool in the study of the space $L_1(m)$ is the \emph{integration operator} $I_m: L^1(m)\to X$, which is the canonical map defined by 
$$
	I_m(f):=\int_\Omega f\, dm
	\quad\text{for all $f\in L_1(m)$.}
$$ 
Note that $\|I_m\|=1$ (see, e.g., \cite[p.~152]{oka-alt}).
We may of course also look at the integration operator defined on $L_\infty(m)$, i.e., the composition 
$$
	I^{(\infty)}_m:=I_m\circ\alpha_\infty:L_\infty(m)\to X,
	\qquad
	I^{(\infty)}_m(f)=\int_\Omega f\, dm
	\quad\text{for all $f\in L_\infty(m)$.}
$$ 
The operator $I^{(\infty)}_m$ is thus weakly compact and $\|I^{(\infty)}_m\|=\|m\|(\Omega)$.

\subsection{The isometric version of the Davis-Figiel-Johnson-Pelcz\'{y}nski procedure}\label{subsection:dfjp}

Let us quickly recall the main construction and results from \cite{dav-alt} together with the extra information obtained in~\cite{lim-nyg-oja}. 

Let $K \sub B_X$ be a closed absolutely convex set and fix $a\in (1,\infty)$. For each $n\in \N$, define the bounded absolutely convex set
\[
	K_n:=a^n K+a^{-n} B_X
\]
and denote by $\|\cdot\|_n$ the Minkowski functional defined by $K_n$, i.e.,
$$
	\|x\|_n:=\inf\{t>0: x\in tK_n\}
	\quad
	\text{for all $x\in X$.}
$$
Note that each $\|\cdot\|_n$ is an equivalent norm on~$X$. The following statements now hold:
\begin{enumerate}
\item[(i)] $X_K:=\{x\in X: \ \sum_{n=1}^\infty \|x\|_n^2 <\infty\}$
is a Banach space equipped with the norm 
$$
	\|x\|_K:=\sqrt{\sum_{n=1}^\infty \|x\|_n^2}.
$$
\item[(ii)] The identity map $J_K: X_K \to X$ is an operator with $\|J_K\|\leq \frac{1}{f(a)}$ and $K \sub f(a) B_{X_K}$, where
$$
	f(a):=\sqrt{\sum_{n=1}^\infty\left(\frac{a^n}{a^{2n}+1}\right)^2}.
$$
\item[(iv)] $J_K^{**}$ is injective (equivalently, $J^*_K(X^*)$ is norm-dense in $X_K^*$).
\item[(v)] For each $x\in K$ we have
\begin{equation}\label{eqn:power2}
	\|x\|_K^2 \leq \Big(\frac{1}{4}+\frac{1}{2\ln a}\Big) \|x\|. \tag{$K^2$}
\end{equation}
\item[(vi)] $J_{K}$ is a norm-to-norm homeomorphism when restricted to~$K$.
\item[(vii)] $J_K$ is a weak-to-weak homeomorphism when restricted to $B_{X_K}$.
\item[(viii)] $X_K$ is reflexive if and only if $K$ is weakly compact.
\end{enumerate}

Given a Banach space~$Z$ and a (non-zero) operator $T:Z\to X$, the previous procedure applied to
$K:=\frac{1}{\|T\|}\overline{T(B_Z)}$ gives a factorization
\begin{equation}\label{eqn:DFJPLNO}
	\xymatrix@R=3pc@C=3pc{Z
	\ar[r]^{T} \ar[d]_{T_K} & X\\
	X_K  \ar@{->}[ur]_{J_K}  & \\
	}
\end{equation}
where $T_K$ is an operator with $\|T_K\|\leq f(a)\|T\|$. The following statements hold:
\begin{enumerate}
\item[(ix)] $X_K$ is reflexive if and only if $T$ is weakly compact if and only if $T_K$ is weakly compact if and only if $J_K$ is weakly compact. 
\item[(x)] $T$ is compact if and only if $T_K$ is compact if and only if $J_K$ is compact. In this case, $X_K$ is separable.
\end{enumerate}

Let $\bar a$ be the unique element of $(1,\infty)$ such that 
$f(\bar a)=1$. When the previous method is performed with $a=\bar a$, we have
$\|T_K\|=\|T\|$ and $\|J_K\|=1$, and \eqref{eqn:DFJPLNO} is called the {\em DFJP-LNO factorization} of~$T$. 

\subsection{An observation on strong measurability}

The next lemma will be needed in the proof of Theorem~\ref{theo:Im}.

\begin{lem}\label{lem:function}
Let $K\sub B_X$ be a closed absolutely convex set and $a\in (1,\infty)$. Let $G:\Omega \to X$ be a function with $G(\Omega)\sub K$, 
$F:\Omega \to X_K$ be the function such that $J_K\circ F=G$, and $\mu$ be a non-negative finite measure on~$\Sigma$.
If $G$ is strongly $\mu$-measurable, then so is~$F$.
\end{lem}
\begin{proof} 
Since $G$ is strongly $\mu$-measurable, there is $E\in \Sigma$ with $\mu(\Omega\setminus E)=0$ such that $G(E)$ is separable. 
Since $G(E) \sub K$, we have $F(E) \sub f(a) B_{X_K}$ (by~(ii)). The separability of~$G(E)$ and~(vii) imply that 
$F(E)$ is separable. On the other hand, $y^*\circ F$ is $\mu$-measurable for every $y^*\in J_K^*(X^*)$
(i.e., $G$ is scalarly $\mu$-measurable) and so the norm-density of $J_K^*(X^*)$ in~$X_K^*$ (property~(iv)) implies that
$F$ is scalarly $\mu$-measurable. An appeal to Pettis' measurability theorem 
(see, e.g., \cite[p.~42, Theorem~2]{die-uhl-J})
ensures that $F$ is strongly $\mu$-measurable.
\end{proof}

\begin{rem}\label{rem:2ndproof}
\rm The strong $\mu$-measurability of a Banach space-valued function~$h$ defined on a finite measure space $(\Omega,\Sigma,\mu)$ is characterized as follows: 
for each $\epsilon>0$ and each $A \in \Sigma$ with $\mu(A)>0$ there is $B \sub A$, $B\in \Sigma$ with $\mu(B)>0$,
such that $\|h(\omega)-h(\omega')\|\leq \epsilon$ for all $\omega,\omega'\in B$ (this is folklore, see \cite[Proposition~2.2]{cas-kad-rod-3}
for a sketch of proof). This characterization and the inequality
$$
	\|x-x'\|_K^2 \leq \Big(\frac{1}{2}+\frac{1}{\ln a}\Big) \|x-x'\|
	\quad
	\text{for all $x,x'\in K$}
$$
(which follows from~\eqref{eqn:power2} and the absolute convexity of~$K$) 
can be combined
to give another proof of Lemma~\ref{lem:function}.
\end{rem}

\section{General factorization results}\label{section:lemmas}

The following lemma is surely folklore to experts in vector measure theory, 
but as its proof requires some tools we provide a proof for the reader's convenience.

\begin{lem}\label{lem:dense}
Let $Y$ be a Banach space, $\Gamma \sub Y^*$ be a norm-dense set, and $\nu:\Sigma \to Y$ be a map such that $y^*\nu$ is 
countably additive for every $y^*\in \Gamma$. Then $\nu$ is a countably additive vector measure.
\end{lem}
\begin{proof}
By the Orlicz-Pettis theorem (see, e.g., \cite[p.~22, Corollary~4]{die-uhl-J}), in order to prove that $\nu$
is countably additive it suffices to show that $y^*\nu$ is countably additive for every $y^*\in Y^*$.
Since $\Gamma$ separates the points of~$Y$, $\nu$ is finitely additive and has bounded range, by the Dieudonn\'{e}-Grothendieck theorem 
(see, e.g., \cite[p.~16, Corollary~3]{die-uhl-J}). 
Fix $y^*\in Y^*$. Let $(B_n)$ be a disjoint sequence in~$\Sigma$
and fix $\epsilon>0$. We can choose $y_0^*\in \Gamma$ such that $|y^*(\nu(A))-y_0^*(\nu(A))|\leq \epsilon$
for all $A\in \Sigma$. Since $y_0^* \nu$ is countably additive, there is $n_0\in \N$ such that  $|y_0^*(\nu(\bigcup_{n\geq n_1}B_n))|\leq \epsilon$ 
for every $n_1\geq n_0$. It follows that $|y^*(\nu(\bigcup_{n\geq n_1}B_n))|\leq 2\epsilon$ for every $n_1\geq n_0$.
This shows that $y^*\nu$ is countably additive.
\end{proof}

\begin{rem}\label{rem:Rainwater} \rm 
\begin{enumerate}
\item[(i)] In Lemma~\ref{lem:dense} we do not really need that $\Gamma$ is norm-dense. It suffices that $\Gamma$ is what one could call a 
{\em Rainwater set}, i.e., a set with the following property: if $(y_n)$ is a bounded sequence in~$Y$ and there is 
$y\in Y$ with $\ys(y_n)\to \ys(y)$ for every $\ys\in\Gamma$, then $y_n\to y$ weakly. 
See the proof of \cite[Proposition~2.9]{fer-alt-4}. The most general known Rainwater sets are (I)-generating sets, in particular James boundaries 
like the extreme points of~$B_{\Ys}$; see \cite{nyg4} for the fact that (I)-generating sets are Rainwater (this was proved independently by Kalenda, 
private communication) and \cite[Theorem~2.3]{fon-lin} for the deep result that James boundaries are (I)-generating.

\item[(ii)] If $Y$ contains no isomorphic copy of~$\ell_\infty$, then the assertion of Lemma~\ref{lem:dense} holds
for any set $\Gamma \sub Y^*$ which separates the points of~$Y$, by a result of Diestel and Faires~\cite{die-fai}
(cf., \cite[p.~23, Corollary~7]{die-uhl-J}).
\end{enumerate}
\end{rem}

The following lemma is essentially known (see, e.g., \cite[Lemma~3.27]{oka-alt}). We add 
an estimate for the norm of the inclusion operator.

\begin{lem}\label{lem:factorizationVM}
Suppose that $m$ factors as 
$$
	\xymatrix@R=3pc@C=3pc{\Sigma
	\ar[r]^{m} \ar[d]_{\tilde{m}} & X\\
	Y  \ar@{->}[ur]_{J}  & \\
	}
$$
where $Y$ is a Banach space, $\tilde{m}$ is a countably additive vector measure and $J$ is an injective operator. Then:
\begin{enumerate}
\item[(i)] $\mathcal{N}(m)=\mathcal{N}(\tilde m)$. 
\item[(ii)] $L_1(\tilde m)$ embeds continuously into $L_1(m)$ with norm $\leq \|J\|$, i.e., 
the identity map $L_1(\tilde m) \to L_1(m)$ is an injective operator with norm $\leq \|J\|$.
\item[(iii)] $I^{(\infty)}_m=J\circ I^{(\infty)}_{\tilde{m}}$.
\end{enumerate}
\end{lem}
\begin{proof} (i) The equality $\mathcal{N}(m)=\mathcal{N}(\tilde m)$
follows at once from the injectivity of~$J$. 

(ii) If $h$ is any $\tilde{m}$-integrable function, then $h$ is $m$-integrable and the equality
$J(\int_\Omega h \, d\tilde{m})=\int_\Omega h \, dm$ holds
(see, e.g., \cite[Lemma~3.27]{oka-alt}). Therefore, for each $f\in L_1(\tilde{m})$, we can apply Fact~\ref{fact:norming} twice
to get
\begin{multline*}
	\|f\|_{L_1(m)}=\sup_{g\in B_{L_\infty(m)}}
	\Big\| \int_\Omega fg \, dm \Big\| = \sup_{g\in B_{L_\infty(m)}}
	\Big\|J\Big(\int_\Omega fg \, d\tilde{m}\Big)\Big\|  \\ \leq \|J\| \sup_{g\in B_{L_\infty(m)}}
	\Big\|\int_\Omega fg \, d\tilde{m}\Big\|=\|J\| \|f\|_{L_1(\tilde{m})}.
\end{multline*} 

(iii) follows from the density of ${\rm sim}\, \Sigma$ in $L_\infty(m)$ and the equality $m=J\circ \tilde{m}$.
\end{proof}

\begin{rem}\label{rem:bddvar}
\rm In the setting of the previous lemma: 
$$
	|m|(A)\leq \|J\| |\tilde{m}|(A)
	\quad\text{for all $A\in \Sigma$.} 
$$
In particular, $m$
has finite variation whenever $\tilde{m}$ does. The converse fails in general, as we pointed out in the introduction.
\end{rem}

Combining Lemmata~\ref{lem:dense} and~\ref{lem:factorizationVM} leads to a factorization result for $I^{(\infty)}_m$ that applies 
to the DFJP factorization procedure:

\begin{cor}\label{cor:factorizationIinfty}
Suppose that $I^{(\infty)}_m$ factors as 
$$
	\xymatrix@R=3pc@C=3pc{L_\infty(m)
	\ar[r]^{I^{(\infty)}_m} \ar[d]_{T} & X\\
	Y  \ar@{->}[ur]_{J}  & \\
	}
$$
where $Y$ is a Banach space, $T$ and $J$ are operators and $J^*(X^*)$ is norm-dense in~$Y^*$.
Define $\tilde{m}:\Sigma\to Y$ 
by $\tilde{m}(A):=T(\chi_A)$ for all $A\in \Sigma$. Then:
\begin{enumerate}
\item[(i)] $\tilde{m}$ is a countably additive vector measure and $m=J\circ \tilde{m}$.
\item[(ii)] $\mathcal{N}(m)=\mathcal{N}(\tilde{m})$. 
\item[(iii)] $L_1(\tilde m)$ embeds continuously into $L_1(m)$ with norm $\leq \|J\|$.
\item[(iv)] $T=I^{(\infty)}_{\tilde{m}}$.
\end{enumerate}
\end{cor}
\begin{proof}
(i) follows from Lemma~\ref{lem:dense} applied to $\Gamma:=J^*(X^*)$ and the countable additivity of~$m$.
(ii) and (iii) follow from Lemma~\ref{lem:factorizationVM}, since $J$ is injective.
Finally, (iii) is a consequence of the continuity of both $T$ and $I^{(\infty)}_{\tilde{m}}$, the density of ${\rm sim}\, \Sigma$ in~$L_\infty(m)$, 
and the fact that $\int_\Omega h \, d\tilde{m}=T(h)$ for every $h\in {\rm sim}\, \Sigma$.
\end{proof}

An isometric version of our next result
was proved in \cite[Lemma~6]{lip}. We include a similar proof for the sake of completeness.

\begin{lem}\label{lem:Lipecki}
Let $Y$ be a Banach space and $\tilde{m}:\Sigma\to Y$ be a countably additive vector measure
with $\mathcal{N}(m)=\mathcal{N}(\tilde{m})$. Suppose that there is a constant $D>0$ such that 
$\|f\|_{L_1(m)} \leq D \|f\|_{L_1(\tilde{m})}$ for every $f\in {\rm sim}\, \Sigma$.
Then $L_1(\tilde{m})$ embeds continuously into $L_1(m)$ with norm $\leq D$.
\end{lem}
\begin{proof}
Consider ${\rm sim}\, \Sigma$ as a linear subspace of $L_1(\tilde{m})$.
By the assumptions, the identity map $i: {\rm sim} \, \Sigma \to L_1(m)$ is well-defined, linear and 
continuous, with norm $\|i\| \leq D$. Since ${\rm sim}\, \Sigma$ is
dense in $L_1(m)$, we can extend $i$ to an operator 
$$
	j: L_1(\tilde{m}) \to L_1(m)
$$
with $\|j\|=\|i\| \leq D$. We claim that $j(f)=f$ for every $f\in L_1(\tilde{m})$. 
Indeed, choose a sequence $(f_n)$ in ${\rm sim}\, \Sigma$ 
such that
$\|f_n- f\|_{L_1(\tilde{m})}\to 0$. By passing to a subsequence, not relabeled, we can assume that $f_n\to f$ $\tilde{m}$-a.e.
(Fact~\ref{fact:convergence}). Since $j$ is an operator and $j(f_n)=f_n$ for all $n\in \N$, we have $\|f_n- j(f)\|_{L_1(m)}\to 0$. Another appeal to Fact~\ref{fact:convergence}
allows us to extract a further subsequence $(f_{n_k})$ such that $f_{n_k} \to j(f)$ $m$-a.e. It follows that $j(f)=f$.     
\end{proof}

The following result first appeared as \cite[Lemma~2.2]{oka-ric} (with a different and simpler proof in \cite[Lemma~3.1]{rod15}), 
but under the extra assumption that $Y$ contains no isomorphic copy of $\ell_\infty$ (in order to prove the first part of (iii) below). 
In addition to removing this unnecessary condition, we also provide explicit estimates for the norm of the identity as a Banach lattice isomorphism between 
$L_1(m)$ and~$L_1(\tilde{m})$. 

\begin{theo}\label{theo:OR}
Suppose that $I_m$ factors as 
$$
	\xymatrix@R=3pc@C=3pc{L_1(m)
	\ar[r]^{I_m} \ar[d]_{T} & X\\
	Y  \ar@{->}[ur]_{J}  & \\
	}
$$
where $Y$ is a Banach space, $T$ and $J$ are operators and $J$ is injective. 
Define $\tilde{m}:\Sigma\to Y$ 
by $\tilde{m}(A):=T(\chi_A)$ for all $A\in \Sigma$. 
Then:
\begin{enumerate}
\item[(i)] $\tilde{m}$ is a countably additive vector measure and $m=J\circ \tilde{m}$.
\item[(ii)] $\mathcal{N}(m)=\mathcal{N}(\tilde{m})$.
\item[(iii)] $L_1(\tilde{m})=L_1(m)$ with equivalent norms. In fact, we have 
$$
	\|T\|^{-1} \|f\|_{L_1(\tilde{m})}  \leq \|f\|_{L_1(m)} \leq \|J\|\|f\|_{L_1(\tilde{m})}
	\quad\text{for every $f\in L_1(m)$.}
$$
\item[(iv)] $T=I_{\tilde{m}}$.
\end{enumerate}
\end{theo}
\begin{proof} (i) Clearly, $\tilde{m}$ is finitely additive. Now, its countable additivity
follows from that of~$m$ and the inequality $\|\tilde{m}(A)\| \leq \|T\| \|\chi_A\|_{L_1(m)}=\|T\| \|m\|(A)$, for all $A\in \Sigma$.
The equality $m=J\circ \tilde{m}$ is obvious.

Lemma~\ref{lem:factorizationVM} implies (ii) and the fact that any $f\in L_1(\tilde{m})$ belongs to~$L_1(m)$, with
$\|f\|_{L_1(m)}\leq \|J\|\|f\|_{L_1(\tilde{m})}$.

On the other hand, observe that for any $h \in {\rm sim}\,\Sigma$ we have $\int_\Omega h \, d\tilde{m}=T(h)$. 
Now, given any $f \in {\rm sim}\,\Sigma$, an appeal to Fact~\ref{fact:norming} yields
\begin{multline*}
	\|f\|_{L_1(\tilde{m})}=\sup_{g\in B_{L_\infty(m)}\cap {\rm sim}\,\Sigma}\Big\|\int_\Omega fg \, d\tilde{m}\Big\|
	=\sup_{g\in B_{L_\infty(m)}\cap {\rm sim}\,\Sigma}\|T(fg)\| \\ \leq \|T\| 
	\sup_{g\in B_{L_\infty(m)}\cap {\rm sim}\,\Sigma}\|fg\|_{L_1(m)} \leq \|T\| \|f\|_{L_1(m)}.
\end{multline*} 

It follows that for every $f\in {\rm sim}\, \Sigma$ we have
\begin{equation}\label{eqn:simple}
	\|T\|^{-1} \|f\|_{L_1(\tilde{m})}  \leq \|f\|_{L_1(m)} \leq \|J\|\|f\|_{L_1(\tilde{m})}.
\end{equation}
We can now apply Lemma~\ref{lem:Lipecki} twice to conclude that $L_1(m)=L_1(\tilde{m})$ 
and that \eqref{eqn:simple} holds for every $f\in L_1(m)$.

Finally, (iv) follows from the continuity of both $T$ and $I_{\tilde{m}}$, the density of ${\rm sim}\,\Sigma$ in~$L_1(\tilde{m})$, 
and the fact that $\int_\Omega h \, d\tilde{m}=T(h)$ for every $h\in {\rm sim}\,\Sigma$.
\end{proof}

\subsection{An observation on vector measures with a Bochner derivative with respect to its variation}

It is known that $m$ has finite variation and a Bochner derivative with respect to~$|m|$ if and only if
$I^{(\infty)}_m$ is nuclear (see, e.g., \cite[p.~173, Theorem~4]{die-uhl-J}). Recall that an operator $T$ from a Banach space~$Z$ to~$X$ is said to be {\em nuclear}
if there exist sequences $(z^*_n)$ in $Z^*$ and $(x_n)$ in~$X$ with $\sum_{n\in \N}\|z^*_n\|\|x_n\|<\infty$
such that $T(z)=\sum_{n\in \N}z^*_n(z) x_n$ for all $z\in Z$.

\begin{theo}\label{theo:RP}
If $m$ has finite variation and a Bochner derivative with respect to~$|m|$, then there exist
a separable reflexive Banach space~$Y$, a countably additive vector measure $\tilde{m}:\Sigma\to Y$ having finite variation and an injective compact operator $J: Y \to X$
such that $m=J\circ \tilde{m}$.
\end{theo}
\begin{proof}
Since $I_m^{(\infty)}$ is a nuclear operator, it can be factored as
$I_m^{(\infty)}=V\circ U$, where $U:L_\infty(m)\to \ell_1$ is a nuclear operator and $V:\ell_1\to X$ is a compact operator (see, e.g., \cite[Proposition~5.23]{die-alt}). 
Now, we can consider the DFJP factorization of~$V$ to obtain the commutative diagram
$$
	\xymatrix@R=3pc@C=3pc{L_\infty(m)
	\ar[r]^{I^{(\infty)}_m} \ar[d]_{U} & X\\
	\ell_1  \ar[r]^{T} \ar@{->}[ur]_{V}  & Y \ar[u]_{J} \\
	}
$$
where $Y$ is a separable reflexive Banach space, $T$ and $J$ are compact operators
and $J^*(X^*)$ is norm-dense in~$Y^*$. 
By Corollary~\ref{cor:factorizationIinfty}, the map $\tilde{m}:\Sigma\to Y$ 
defined by $\tilde{m}(A):=(T\circ U)(\chi_A)$ for all $A\in \Sigma$ is a countably additive vector measure
such that $\mathcal{N}(m)=\mathcal{N}(\tilde m)$, $m=J\circ \tilde m$ and $I_{\tilde{m}}^{(\infty)}=T\circ U$. Since $U$ is nuclear, the same holds for $I_{\tilde{m}}^{(\infty)}$,
and so $\tilde{m}$ has finite variation.
\end{proof}

\section{DFJP-LNO factorization of integration operators}\label{section:Iinfty}

The following theorem collects some consequences of the DFJP-LNO factorization 
when applied to~$I_m^{(\infty)}$.

\begin{theo}\label{theo:Iinfty}
Let us consider the DFJP-LNO factorization of~$I^{(\infty)}_m$, as follows
$$
	\xymatrix@R=3pc@C=3pc{L_\infty(m)
	\ar[r]^{I^{(\infty)}_m} \ar[d]_{T} & X\\
	Y  \ar@{->}[ur]_{J}  & \\
	}
$$
Let $\tilde{m}:\Sigma \to Y$ be the countably additive vector measure defined
by $\tilde{m}(A):=T(\chi_A)$ for all $A\in \Sigma$ (see Corollary~\ref{cor:factorizationIinfty}). Then:
\begin{enumerate}
\item[(i)] $Y$ is reflexive.
\item[(ii)] $L_1(\tilde m)$ embeds continuously into $L_1(m)$ with norm $\leq 1$.
\item[(iii)] $\|\tilde{m}\|(\Omega)=\|m\|(\Omega)$.
\item[(iv)] There is a universal constant $C>0$ such that
\begin{equation}\label{eqn:21infty}
	\|I^{(\infty)}_{\tilde{m}}(f)\|^2 \leq C \|m\|(\Omega) \|f\|_{L_\infty(m)} \|I^{(\infty)}_m(f)\|
	\quad\text{for every $f\in L_\infty(m)$.}
\end{equation}
In particular, $\|\tilde{m}(A)\|^2 \leq C \|m\|(\Omega) \|m(A)\|$ for all $A\in \Sigma$.
\end{enumerate}
\end{theo}
\begin{proof}
The DFJP-LNO factorization is done by using the set
$$
	K:=\frac{1}{\|m\|(\Omega)}\overline{I^{(\infty)}_m(B_{L_\infty(m)})},
$$
so that $T=T_K$, $J=J_K$, $\|T\|=\|I_m^{(\infty)}\|=\|m\|(\Omega)$ and $\|J\|=1$. Since $K$ is weakly compact, $Y$ is reflexive.
Statements~(ii) and~(iii) follow from Corollary~\ref{cor:factorizationIinfty}, bearing in mind that
$\|I_{\tilde{m}}^{(\infty)}\|=\|\tilde{m}\|(\Omega)$. Finally, (iv) follows from inequality~\eqref{eqn:power2}, which 
implies $\|I^{(\infty)}_{\tilde{m}}(g)\|^2 \leq C \|m\|(\Omega) \|I^{(\infty)}_m(g)\|$
for every $g\in B_{L_\infty(m)}$. 
\end{proof}

\begin{rem}\label{rem:Iinfty}
\rm In the setting of Theorem~\ref{theo:Im}:
\begin{enumerate}
\item[(i)] $|m|(A)\leq |\tilde{m}|(A)$ for all $A\in \Sigma$ (because $\|J\|=1$).
\item[(ii)] $\mathcal{R}(m)$ is relatively norm-compact if and only if $I^{(\infty)}_m$ is compact. In this case,
$\mathcal{R}(\tilde{m})$ is relatively norm-compact
(because the restriction $J|_K$ is a norm-to-norm homeomorphism and $\mathcal{R}(m) \sub \|m\|(\Omega) K$)
and $Y$ is separable.
\item[(iii)] $L_1(\tilde{m})$ is weakly sequentially complete because $Y$ contains no isomorphic copy of~$c_0$,
see \cite[Theorem~3]{cur1} (cf., \cite{cal-alt-7,oka-alt5}). 
In general, $L_1(m)$ is not weakly sequentially complete, so the equality $L_1(\tilde{m})=L_1(m)$ can fail. 
\end{enumerate}
\end{rem}

An operator $T$ from a Banach space~$Z$ to~$X$ is said to be {\em $(2,1)$-summing} if $\sum_{n\in \N}\|T(z_n)\|^2<\infty$ for every 
unconditionally convergent series $\sum_{n\in \N}z_n$ in~$Z$. A glance at inequality~\eqref{eqn:21infty} reveals that $I^{(\infty)}_{\tilde{m}}$
is $(2,1)$-summing whenever $I^{(\infty)}_{m}$ is $1$-summing, which in turn is equivalent to saying that $m$ has finite variation 
(see, e.g., \cite[Corollary~4, p.~164]{die-uhl-J}).

On the other hand, a result of Pisier~\cite{pis2} (cf., \cite[Theorem~10.9]{die-alt}) characterizes
$(2,1)$-summing operators from a $C(K)$ space (like $L_\infty(m)$) to a Banach space as those operators which factor through a Lorentz space $L_{2,1}(\mu)$ 
for some regular Borel probability on~$K$, via the canonical map from $C(K)$ to $L_{2,1}(\mu)$; the hardest part of this result consists in obtaining an inequality
similar to~\eqref{eqn:21infty}. 

Inequality~\eqref{eqn:21infty} 
and the easier part of Pisier's argument can be combined to obtain that $I^{(\infty)}_{\tilde{m}}$
can be extended to $L_{2,1}(|m|)$ whenever $m$ has finite variation. In fact,  
$I^{(\infty)}_{\tilde{m}}$ can always be extended to the Lorentz type space $L_{2,1}(\|m\|)$ associated to the semivariation of~$m$,
no matter whether $m$ has or not finite variation. We include this result
in Proposition~\ref{pro:Pisier} below. Its proof is omitted since
it can be done just by imitating some parts of the proof of \cite[Theorem~10.9]{die-alt}.
Let us recall that $L_{2,1}(\|m\|)$ is the set of all ($m$-a.e. equivalence classes of) $\Sigma$-measurable functions $f:\Omega \to \erre$
for which 
$$
	\|f\|_{L_{2,1}(\|m\|)}:=2\int_0^\infty \sqrt{\|m\|_f(t)} \, dt < \infty,
$$
where $\|m\|_f$ is the distribution function of~$f$ with respect to~$\|m\|$, defined by $\|m\|_f(t):=\|m\|(\{\omega \in \Omega: |f(\omega)|>t\})$ 
for all $t>0$. The linear space $L_{2,1}(\|m\|)$ is a Banach lattice when equipped with a certain norm which is 
equivalent to the quasi-norm $\|\cdot\|_{L_{2,1}(\|m\|)}$. 
In general, $L_\infty(m) \sub L_{2,1}(\|m\|) \sub L_1(m)$ with continuous inclusions. 
Note that if $m$ has finite variation, then 
the Lorentz space $L_{2,1}(|m|)$ embeds continuously into $L_{2,1}(\|m\|)$.
The Lorentz spaces associated to the semivariation of a vector measure
were introduced in~\cite{fer-alt-6} and studied further in~\cite{cam-alt-5}.

\begin{pro}\label{pro:Pisier}
In the setting of Theorem~\ref{theo:Iinfty}, $I^{(\infty)}_{\tilde{m}}$ factors as
$$
	\xymatrix@R=3pc@C=3pc{L_\infty(m) 
	\ar@{->}[dr]_{I^{(\infty)}_{\tilde{m}}}
	\ar[r]^{I^{(\infty)}_m} \ar[d]_{i} & X\\
	 L_{2,1}(\|m\|)  \ar[r]^{S}   & Y \ar[u]_{J} \\
	}
$$
where $i$ is the identity operator and $S$ is an operator. In particular, if $m$ has finite variation, then 
$I^{(\infty)}_{\tilde{m}}$ can be extended to the Lorentz space $L_{2,1}(|m|)$.
\end{pro}

\begin{rem}\label{rem:Iminfty-lattice}
\rm Suppose that $X$ is a Banach lattice. If we apply the DFJP-LNO method to the closed convex solid hull $K_0$ of
$\frac{1}{\|m\|(\Omega)}I^{(\infty)}_m(B_{L_\infty(m)})$, then we get a factorization as
$$
	\xymatrix@R=3pc@C=3pc{L_\infty(m)
	\ar[r]^{I^{(\infty)}_m} \ar[d]_{T_{K_0}} & X\\
	Y_{K_0}  \ar@{->}[ur]_{J_{K_0}}  & \\
	}
$$
where $Y_{K_0}$ is a Banach lattice and both $J_{K_0}$ and $J^*_{K_0}$
are interval preserving lattice homomorphisms
(imitate the proof of \cite[Theorem~5.41]{ali-bur}). 
Moreover:
\begin{enumerate}
\item[(i)] The statements of Theorem~\ref{theo:Iinfty}, Remark~\ref{rem:Iinfty}(i)-(ii) and Proposition~\ref{pro:Pisier}
also hold for this factorization, with the exception that $Y_{K_0}$ need not be reflexive.
\item[(ii)] $Y_{K_0}$ is reflexive whenever $X$ has the property that the solid hull of 
any relatively weakly compact set is relatively weakly compact. 
This happens if either $X$ contains no isomorphic copy of~$c_0$ (see, e.g., \cite[Theorems~4.39 and 4.60]{ali-bur})
or $X$ is order continuous and atomic, see \cite[Theorem~2.4]{che-wic}.
\item[(iii)] $\tilde m$ is positive and $Y_{K_0}$ is reflexive whenever $m$ is positive. Indeed, in this case an easy computation shows that
$[-m(\Omega),m(\Omega)]$ is the solid hull of~$\mathcal{R}(m)$ and that
$$
	K_0=\frac{1}{\|m\|(\Omega)}[-m(\Omega),m(\Omega)].
$$ 
Now, from \cite[Theorem~2.4]{fer-nar} it follows that $K_0$ is $L$-weakly compact 
and so it is weakly compact (see, e.g., \cite[Proposition~3.6.5]{mey2}). 
 
\item[(iv)] In general, $Y^*_{K_0}$ contains no isomorphic copy of~$c_0$, because
$I^{(\infty)}_m$ is weakly compact (imitate the proof of \cite[Theorem~5.43]{ali-bur}).
\end{enumerate}
\end{rem}

We next apply the DFJP-LNO factorization procedure to $I_m$. The following theorem 
gathers some consequences of it.

\begin{theo}\label{theo:Im}
Let us consider the DFJP-LNO factorization of~$I_m$, as follows
$$
	\xymatrix@R=3pc@C=3pc{L_1(m)
	\ar[r]^{I_m} \ar[d]_{T} & X\\
	Y  \ar@{->}[ur]_{J}  & \\
	}
$$
Let $\tilde{m}:\Sigma\to Y$ be the countably additive vector measure defined by 
$\tilde{m}(A):=T(\chi_A)$ for all $A\in \Sigma$ (see Theorem~\ref{theo:OR}). Then:
\begin{enumerate}
\item[(i)] $L_1(\tilde{m})=L_1(m)$ with equal norms, i.e.,
$$
	\|f\|_{L_1(m)}=\|f\|_{L_1(\tilde{m})}
	\quad\text{for all $f\in L_1(m)$.}
$$
In particular, $\|m\|(A) = \|\tilde{m}\|(A)$ for all $A\in \Sigma$.
\item[(ii)] There is a universal constant $C>0$ such that
\begin{equation}\label{eqn:21}
	\|I_{\tilde{m}}(f)\|^2 \leq C \|f\|_{L_1(m)} \|I_m(f)\| 
	\quad\text{for every $f\in L_1(m)$.}
\end{equation}
In particular, $\|\tilde{m}(A)\|^2 \leq C \|m\|(A) \|m(A)\|$
for all $A\in \Sigma$.
\item[(iii)] $|m|(A)\leq |\tilde{m}|(A) \leq \sqrt{C} |m|(A)$ for all $A\in \Sigma$. Therefore,
$\tilde{m}$ has finite (resp., $\sigma$-finite) variation whenever $m$ does.
\item[(iv)] If $m$ has finite variation and a Bochner derivative $G$ with respect to~$|m|$, then 
$\tilde{m}$ has a Bochner derivative $\tilde{F}$ with respect to $|\tilde{m}|$ and 
$$
	\int_\Omega \|\tilde{F}\|^2 \, d|\tilde{m}|
	\leq C \int_\Omega \|G\| \, d|m|.
$$
\end{enumerate}
\end{theo}
\begin{proof} The factorization is done by using the set $K:=\overline{I_m(B_{L_1(m)})}$,
so that $T=T_K$, $J=J_K$ and $\|T\|=\|J\|=1$.

(i) follows from Theorem~\ref{theo:OR}, while (ii) is consequence of inequality~\eqref{eqn:power2}, which 
in this case reads as $\|I_{\tilde{m}}(g)\|^2 \leq C \|I_m(g)\|$ for every $g\in B_{L_1(m)}$, where we write
$C:=\frac{1}{4}+\frac{1}{2\ln \bar a}$ and $\bar a$ is as in Subsection~\ref{subsection:dfjp}.

(iii) Fix $A\in \Sigma$. The inequality $|m|(A)\leq |\tilde{m}|(A)$ follows at once from
the fact that $\|J\|=1$. On the other hand, the inequality $|\tilde{m}|(A)\leq \sqrt{C}|m|(A)$ is obvious
if $|m|(A)$ is infinite, so we assume that $|m|(A)<\infty$. 
Now, given finitely many pairwise disjoint $A_1,\dots,A_n \in \Sigma$ with $A_i\sub A$, we have
\begin{equation}\label{eqn:2varBIS}
	\sum_{i=1}^n \frac{\|\tilde{m}(A_i)\|^2}{|m|(A_i)}
	\leq
	\sum_{i=1}^n \frac{\|\tilde{m}(A_i)\|^2}{\|m\|(A_i)}
	\stackrel{\text{(ii)}}{\leq}
	C \sum_{i=1}^n \|m(A_i)\| \leq C|m|(A)
\end{equation}
(with the convention $\frac{0}{0}=0$) and so Holder's inequality yields 
$$
	\sum_{i=1}^n \|\tilde{m}(A_i)\|
	 \leq
	\Big( \sum_{i=1}^n \frac{\|\tilde{m}(A_i)\|^2}{|m|(A_i)} \Big)^{1/2} \cdot
	\Big( \sum_{i=1}^n |m|(A_i)\Big)^{1/2} \\ \stackrel{\eqref{eqn:2varBIS}}{\leq} 
	\sqrt{C} |m|(A).
$$
This shows that $|\tilde{m}|(A) \leq \sqrt{C} |m|(A)$.

(iv) Let $G:\Omega \to X$ be a Bochner derivative of $m$ with respect to~$|m|$.
Then there is $E\in \Sigma$ with $|m|(\Omega \setminus E)=0$ such that 
\begin{multline*}
	G(E)\sub H:=\overline{\Big\{\frac{m(A)}{|m|(A)}: \, A\in \Sigma, \, |m|(A)>0 \Big\}}
	 \\ \sub
	\overline{{\rm aco}}\Big(\Big\{\frac{m(A)}{\|m\|(A)}: \, A\in \Sigma, \, \|m\|(A)>0 \Big\}\Big) \sub K
\end{multline*}
(see, e.g., \cite[Lemma~2.3]{kup} or \cite[Lemma 3.7]{cas-kad-rod-5}). We can assume without loss of generality that $E=\Omega$.
Let $F:\Omega \to Y$ be the function satisfying $J\circ F=G$. Then $F$ is 
strongly $|m|$-measurable (by Lemma~\ref{lem:function}). Note that $F$ is bounded (we have $F(\Omega) \sub K \sub B_{Y}$) 
and so $F$ is Bochner integrable with respect to~$|m|$. 
Since $J$ is injective, we have $\int_A F \, d|m|=\tilde{m}(A)$ for all $A\in \Sigma$.  

Note that $|m|$ and $|\tilde{m}|$ have the same null sets, hence $F$ is strongly $|\tilde{m}|$-measurable as well.
Since $F$ is bounded, it is Bochner integrable with respect to~$|\tilde{m}|$. On the other hand,
let $\varphi$ be the Radon-Nikod\'{y}m derivative of~$|m|$ with respect to~$|\tilde{m}|$. Then
$0\leq \varphi \leq 1$ $|\tilde{m}|$-a.e. (because $|m|(A)\leq |\tilde{m}|(A)$ for all $A\in \Sigma$)
and, therefore, the product $\tilde{F}:=\varphi F:\Omega \to Y$ is Bochner integrable with respect to~$|\tilde{m}|$, with integral
$\int_A \tilde F \, d|\tilde{m}|=\int_A F \, d|m|=\tilde{m}(A)$ for all $A\in \Sigma$.

Finally, by~\eqref{eqn:power2} and the inclusion $G(\Omega)\sub K$, we have
$\|F(\omega)\|^2\leq C \|G(\omega)\|$ for every $\omega\in \Omega$, so that
\begin{multline*}
	\int_\Omega \|\tilde F \|^2 \, d|\tilde{m}|=
	\int_\Omega \varphi^2 \|F\|^2 \, d|\tilde{m}|=
	\int_\Omega \varphi \|F\|^2 \, d|m|
	\\ \leq \int_\Omega \|F\|^2 \, d|m|	
	\leq C \int_\Omega \|G\| \, d|m|,
\end{multline*}
as we wanted to prove.
\end{proof}

\begin{rem}\label{rem:Im}
\rm In the setting of Theorem~\ref{theo:Im}, the following statements hold:
\begin{enumerate}
\item[(i)] If $\mathcal{R}(m)$ is relatively norm-compact, then so is $\mathcal{R}(\tilde{m})$
(because $J|_K$ is a norm-to-norm homeomorphism and $\mathcal{R}(m) \sub \|m\|(\Omega) K$).
\item[(ii)] If $I_m$ is compact, then $J$ and $I_{\tilde{m}}$ are compact as well.
\item[(iii)] If $I_m$ is completely continuous, then so is $I_{\tilde{m}}$. This is also an immediate consequence of the fact that $J|_K$
is a norm-to-norm homeomorphism. It was proved in \cite[Lemma~3.2(ii)]{rod15}
for the DFJP factorization, with a rather more complicated proof, under the additional assumption that $Y$ contains no isomorphic copy of~$\ell_\infty$.
\item[(iv)] If $m$ has finite variation, then the $2$-variation of~$\tilde{m}$ with respect to~$|m|$ is finite and less than or equal to $\sqrt{C|m|(\Omega)}$.
This is a consequence of inequality~\eqref{eqn:2varBIS}.
\end{enumerate}
\end{rem}

We arrive at the isometric version of \cite[Proposition~2.1]{oka-ric}
which we already mentioned in the introduction.

\begin{cor}\label{cor:Iweaklycompact} 
Suppose that $I_m$ is weakly compact. Then there exist a reflexive Banach space~$Y$, 
a countably additive vector measure $\tilde{m}:\Sigma\to Y$ and an injective operator $J: Y \to X$
such that $m=J\circ \tilde{m}$ and $L_1(m)=L_1(\tilde{m})$ with equal norms.
\end{cor}

Suppose that $I_m$ is weakly compact. It is known that:
\begin{enumerate}
\item[(i)] If $m$ has finite variation, then the composition of $I_m$
and the continuous embedding of $L_1(|m|)$ into~$L_1(m)$ (see, e.g., \cite[Lemma~3.14]{oka-alt}) is a weakly compact operator and so it is representable 
(see, e.g., \cite[p.~75, Theorem~12]{die-uhl-J}). Hence $m$ admits a Bochner derivative with respect to~$|m|$. 
\item[(ii)] If $m$ has $\sigma$-finite variation, then $\mathcal{R}(m)$ is relatively norm-compact, see \cite[Corollary~3.11]{cal-alt-6}
(cf., \cite[Claim~2]{cur3}). 
\end{enumerate}

The previous statements can be improved as follows.

\begin{cor}\label{cor:clrs}
Suppose that $I_m$ is weakly compact.
\begin{enumerate}
\item[(i)] If $m$ has finite variation, then it admits a Bochner derivative with respect to any control measure.
\item[(ii)] If $m$ has $\sigma$-finite variation, then it admits a strongly measurable Pettis integrable derivative 
with respect to any control measure.
\end{enumerate}
\end{cor}
\begin{proof} Let us consider a factorization of~$I_m$ as in Theorem~\ref{theo:Im}. Since $I_m$ is weakly compact,
$Y$ is reflexive.
Observe that $m$ and $\tilde{m}$ have the same control measures, because $\mathcal{N}(m)=\mathcal{N}(\tilde{m})$. 
Let $\mu$ be a control measure of both $m$ and $\tilde{m}$. 
If $m$ has finite (resp., $\sigma$-finite) variation, then the same holds for~$\tilde{m}$. By the Radon-Nikod\'{y}m property of~$Y$, $\tilde{m}$ has a Bochner
(resp., strongly measurable Pettis) derivative with respect to~$\mu$, say $F:\Omega \to Y$. The composition $J\circ F$ is then 
a Bochner (resp., strongly measurable Pettis) derivative of~$m$ with respect to~$\mu$.
\end{proof}

\begin{rem}\label{rem:Im-lattice}
\rm Suppose that $X$ is a Banach lattice. As in Remark~\ref{rem:Iminfty-lattice}, we can apply the DFJP-LNO procedure to the closed convex solid hull $K_0$ of
$I_m(B_{L_1(m)})$ to obtain a factorization 
$$
	\xymatrix@R=3pc@C=3pc{L_1(m)
	\ar[r]^{I_m} \ar[d]_{T_{K_0}} & X\\
	Y_{K_0}  \ar@{->}[ur]_{J_{K_0}}  & \\
	}
$$
where $Y_{K_0}$ is a Banach lattice and both $J_{K_0}$ and $J^*_{K_0}$
are interval preserving lattice homomorphisms. The statements of Theorem~\ref{theo:Im} and Remark~\ref{rem:Im}
also hold for this factorization. If $I_m$ is weakly compact, then: (i)~$Y_{K_0}^*$ contains no isomorphic copy of~$c_0$
(imitate the proof of \cite[Theorem~5.43]{ali-bur}), and (ii)~$Y_{K_0}$ is reflexive whenever $X$ has the property that 
the solid hull of any relatively weakly compact set is relatively weakly compact.
\end{rem}

%\subsection*{Acknowledgements}
%The research of Jos\'{e} Rodr\'{i}guez was partially supported by {\em Agencia Estatal de Investigaci\'{o}n} [MTM2017-86182-P, grant cofunded by ERDF, EU] 
%and {\em Fundaci\'on S\'eneca} [20797/PI/18].

%\bibliography{C:/Mat/Bibliografia/inves,C:/Mat/Bibliografia/invesJose}

\bibliographystyle{amsplain}

\end{document}